\documentclass[12pt]{article}

\usepackage{amsmath, epsfig, cite, setspace}
\usepackage{amssymb}
\usepackage{amsfonts}
\usepackage{latexsym}
\usepackage{amsthm}

\newtheorem{thm}{Theorem}[section]

\newtheorem{conj}[thm]{Conjecture}
\newtheorem{lem}[thm]{Lemma}

\numberwithin{equation}{section}

\begin{document}

\nocite{*}
\begin{center}
{\Large\bf On a conjecture related to integer-valued polynomials}
\end{center}

\vskip 2mm \centerline{Victor J. W. Guo}
\begin{center}
{\footnotesize School of Mathematics and Statistics, Huaiyin Normal
University, Huai'an, Jiangsu 223300,
 People's Republic of China\\
{\tt jwguo@math.ecnu.edu.cn} }

\end{center}


\vskip 0.7cm \noindent{\bf Abstract.} Using the following $_4F_3$ transformation formula
$$
\sum_{k=0}^{n}{-x-1\choose k}^2{x\choose n-k}^2=\sum_{k=0}^{n}{n+k\choose 2k}{2k\choose k}^2{x+k\choose 2k},
$$
which can be proved by Zeilberger's algorithm, we confirm some special cases of a recent conjecture of Z.-W. Sun on integer-valued polynomials.

\vskip 3mm \noindent {\it Keywords}: Zeilberger's algorithm; Chu-Vandermonde summation; integer-valued polynomials; multi-variable Schmidt polynomials

\vskip 2mm
\noindent{\it MR Subject Classifications}: 33C20, 11A07, 11B65, 05A10

\section{Introduction}
Recall that a polynomial $P(x)\in\mathbb{Q}[x]$ is called {\it integer-valued}, if $P(x)\in\mathbb{Z}$ for all $x\in\mathbb{Z}$.
During the past few years, integer-valued polynomials have been investigated by several authors (see, for example, \cite{Guo-int,Liu,Sun2017}).
Recently, Z.-W. Sun \cite[Conjectures 35(i)]{Sun19} proposed the following conjecture.
\begin{conj}[Z.-W. Sun]\label{conj:sun-1}
Let $l,m,n$ be positive integers and $\varepsilon=\pm1$. Then the polynomial
$$
\frac{1}{n}\sum_{k=0}^{n-1}\varepsilon^k (2k+1)^{2l-1}\sum_{j=0}^{k}{-x-1\choose j}^m{x\choose k-j}^m
$$
is integer-valued.
\end{conj}

By the Chu-Vandermonde summation formula, we have
$$
\sum_{j=0}^{k}{-x-1\choose j}{x\choose k-j}={-1\choose k}=(-1)^k.
$$
Thus, by \cite[Lemmas 2.3 and 2.4]{Mao}, we see that Conjecture \ref{conj:sun-1} is true for $m=1$.
In this note, we shall confirm Conjecture \ref{conj:sun-1}  for $m=2$.

\begin{thm}\label{thm:main-1}
Let $l$ and $n$ be positive integers and $\varepsilon=\pm1$. Then the polynomial
\begin{equation}
\frac{1}{n}\sum_{k=0}^{n-1}\varepsilon^k (2k+1)^{2l-1}\sum_{j=0}^{k}{-x-1\choose j}^2{x\choose k-j}^2 \label{eq:main-1}
\end{equation}
is integer-valued.
\end{thm}

We shall also prove the following result, which confirms the $l=1$ cases of \cite[Conjectures 35(ii)]{Sun19}.
\begin{thm}\label{thm:main-2}
Let $n$ be a positive integer. Then the polynomial
\begin{align}
\frac{1}{n^2}\sum_{k=0}^{n-1}(2k+1)\sum_{j=0}^{k}{-x-1\choose j}^2{x\choose k-j}^2 \label{n2-1}
\end{align}
is integer-valued.
\end{thm}

\section{Proof of Theorem \ref{thm:main-1}}
We first require the following $_4F_3$ transformation formula.
\begin{lem}\label{lem:1}
Let $n$ be a non-negative integer. Then
\begin{equation}
\sum_{k=0}^{n}{-x-1\choose k}^2{x\choose n-k}^2=\sum_{k=0}^{n}{n+k\choose 2k}{2k\choose k}^2{x+k\choose 2k}.
\label{eq:trans}
\end{equation}
\end{lem}
\begin{proof}Denote the left-hand side or the right-hand side of \eqref{eq:trans} by $S_n(x)$. Applying Zeilberger's algorithm (see \cite{AM,PWZ}),
 we obtain
\begin{align*}
(n+2)^3S_{n+2}(x)-(2n+3)(n^2+2x^2+3n+2x+3)S_{n+1}(x)+(3n^2+3n+1)S_n(x)=0.
\end{align*}
That is to say, both sides of \eqref{eq:trans} satisfy the same recurrence relation of order $2$.
Moreover, the two sides of \eqref{eq:trans} are equal for $n=0,1$. This completes the proof.
\end{proof}

\noindent{\it Remark.} Using Zeilberger's algorithm, Z.-W. Sun \cite[Eq. (3.1)]{Sun2012} found the following identity:
\begin{equation}
16^n \sum_{k=0}^{n}{-1/2\choose k}^2{-1/2\choose n-k}^2
=\sum_{k=0}^{n}{2k\choose k}^3{k\choose n-k}(-16)^{n-k}, \label{eq:one}
\end{equation}
and he \cite[Eq. (3.1)]{Sun2014} gave the following formula:
\begin{equation}
64^n \sum_{k=0}^{n}{-1/4\choose k}^2{-3/4\choose n-k}^2
=\sum_{k=0}^{n}{2k\choose k}^3{2n-2k\choose n-k}16^{n-k}.  \label{eq:two}
\end{equation}
Here we point out that, for $x=-1/2$ and $-3/4$, Eq. \eqref{eq:trans} gives identities different from \eqref{eq:one} and \eqref{eq:two}.

In \cite{CG}, Chen and the author introduced the multi-variable Schmidt polynomials
$$
S_{n}(x_0,\ldots,x_{n})=\sum_{k=0}^n {n+k \choose 2k}{2k\choose k} x_k.
$$
In order to prove Theorem \ref{thm:main-1}, we also need the following result, which
is a special case of the last congruence in \cite[Section 4]{CG}.

\begin{lem}\label{lem:2}
Let $l$ and $n$ positive integers and $\varepsilon=\pm1$. Then all the coefficients in
\begin{align*}
\sum_{k=0}^{n-1}\varepsilon^k (2k+1)^{2l-1}S_{k}(x_0,\ldots,x_{k}).
\end{align*}
are multiples of $n$.
\end{lem}

\begin{proof}[Proof of Theorem {\rm\ref{thm:main-1}}]
For any non-negative integer $k$, define
\begin{align*}
x_k &={2k\choose k}{x+k\choose 2k}.
\end{align*}
Then the identity \eqref{eq:trans} may be rewritten as
\begin{align}
\sum_{k=0}^{n}{-x-1\choose k}^2{x\choose n-k}^2=\sum_{k=0}^{n}{n+k\choose 2k}{2k\choose k}x_k.  \label{eq:lem2.2}
\end{align}
It is easy to see that $x_0,\ldots,x_n$ are all integers on condition that $x$ is an integer. By Eq.~\eqref{eq:lem2.2} and Lemma \ref{lem:2}, we see that the
polynomial \eqref{eq:main-1} is integer-valued.
\end{proof}

\section{Proof of Theorem \ref{thm:main-2}}
We need the following result, which can be easily proved by induction on $n$.
See also \cite[Eq. (2.4)]{CG}.
\begin{lem}Let $n$ and $k$ be non-negative integers with $k\leqslant n-1$. Then
\begin{align}
\sum_{m=k}^{n-1}(2m+1){m+k\choose 2k}{2k\choose k}=n{n\choose k+1}{n+k\choose k}. \label{eq:simple}
\end{align}
\end{lem}

\begin{proof}[Proof of Theorem {\rm\ref{thm:main-2}}]
Using the identities \eqref{eq:trans} and \eqref{eq:simple}, we have
\begin{align*}
\sum_{m=0}^{n-1}(2m+1)\sum_{k=0}^{m}{-x-1\choose k}^2{x\choose n-k}^2
&=\sum_{m=0}^{n-1}(2m+1)\sum_{k=0}^{m}{n+k\choose 2k}{2k\choose k}^2{x+k\choose 2k} \\
&=\sum_{k=0}^{n-1}n{n\choose k+1}{n+k\choose k}{2k\choose k}{x+k\choose 2k}.
\end{align*}
It follows that the expression \eqref{n2-1} can be written as
\begin{align}
\sum_{k=0}^{n-1}\frac{1}{n}{n\choose k+1}{n+k\choose k}{2k\choose k}{x+k\choose 2k}
=\sum_{k=0}^{n-1}\frac{1}{k+1}{n-1\choose k}{n+k\choose k}{2k\choose k}{x+k\choose 2k}. \label{eq:xx+1}
\end{align}
Since $\frac{1}{k+1}{2k\choose k}={2k\choose k}-{2k\choose k-1}$ is clearly an integer (the $n$-th Catalan number),
we conclude that the right-hand side of \eqref{eq:xx+1} is also an integer whenever $x$ is an integer. This proves the theorem.
\end{proof}

\section{Concluding remarks}
Z.-W. Sun \cite[Conjecture 35(ii)]{Sun19} conjectured that,
for all positive integers $l$ and $n$, the polynomial
\begin{align*}
\frac{(2l-1)!!}{n^2}\sum_{k=0}^{n-1}(2k+1)^{2l-1}\sum_{j=0}^{k}{-x-1\choose j}^2{x\choose k-j}^2
\end{align*}
is integer-valued. Here $(2l-1)!!=(2l-1)(2l-3)\cdots3\cdot 1$.

We believe that the following (stronger) result is true.
\begin{conj}Let $l$ and $n$ be positive integers and $k$ a non-negative integer with $k\leqslant n-1$. Then
\begin{align}
(2l-1)!!\sum_{m=k}^{n-1}(2m+1)^{2l-1}{m+k\choose 2k}{2k\choose k}^2 \equiv 0\pmod{n^2}. \label{eq:final}
\end{align}
\end{conj}

Our proof of Theorem \ref{thm:main-2} implies that the above conjecture is true for $l=1$.
In view of \eqref{eq:trans}, Sun's conjecture follows from \eqref{eq:final} too.

Recently, $q$-analogues of congruences have been studied by many authors. See \cite{GG,Guo-m3,GL,GS,GuoZu}
and references therein. For $l=1$, we have a $q$-analogue of \eqref{eq:final} as follows:
\begin{align}
\sum_{m=k}^{n-1}[2m+1]{m+k\brack 2k}{2k\brack k}^2 q^{-(k+1)m} \equiv 0\pmod{[n]^2}, \label{eq:q-sun}
\end{align}
where $[n]=1+q+\cdots+q^{n-1}$ is the $q$-integer and ${n\brack k}=\prod_{j=1}^{k}(1-q^{n-k+j})/(1-q^j)$
denotes the $q$-binomial coefficient. The proof of \eqref{eq:q-sun} is similar to that of Theorem \ref{thm:main-2}.
Nevertheless, we cannot found any $q$-analogue of \eqref{eq:final} for $l>1$.

\end{document}